\documentclass[12pt]{article}
\usepackage{amssymb,amsfonts,amsmath,amsthm,tikz}
\usepackage{xcolor}
\usepackage[utf8]{inputenc}
\usepackage{bm}
\usepackage{mathtools}
\usepackage{soul}
\usepackage{enumerate}
\usepackage[english]{babel}
\usepackage{amsthm}
\usepackage{faktor}
\usepackage{enumitem}
\usepackage{hyperref}
\usepackage{mathabx}
\usepackage{flexisym}
\newcommand{\Z}{\mathbb{Z}}
\newcommand{\Q}{\mathbb{Q}}
\newcommand{\R}{\mathbb{R}}
\newcommand{\N}{\mathbb{N}}

\newcommand{\J}{\mathbb{J}}

\newcommand{\0}{\mathcal{O}}
\newcommand{\LO}{\mathcal{LO}}
\newcommand{\Aut}{\mathrm{Aut}}
\newcommand{\BO}{\mathcal{BO}}

\newcommand{\h}{\mathrm{Homeo}_{+}{(\R)}}
\theoremstyle{definition}

\usepackage[lite]{amsrefs}

\theoremstyle{definition}

\newtheorem{theorem}{Theorem}
\newtheorem{lemma}[theorem]{Lemma}
\newtheorem{proposition}{Proposition}

\theoremstyle{remark}
\newtheorem{remark}{Remark}

\theoremstyle{plain}

\theoremstyle{definition}

\theoremstyle{remark}

\begin{document}

	\title{Free products of bi-orderable groups}
	
	\author{Kyrylo Muliarchyk}
	
	\date{}
	
	\maketitle
	
	\begin{abstract}
		We prove a bi-ordered version of Rivas' result for free products of left-order groups. Namely, we show that a free product of bi-ordered groups does not admit isolated bi-ordering. Our method relies on the dynamical realization of bi-ordered groups introduced in \cite{my1}. We also show that the natural action of the automorphism group $Aut(F_2)$ on $F_2$ does not have dense orbits.
	\end{abstract}
	
	\section{Introduction}
	
	\label{sec:1}
	A total relation $<$ on a group $G$ is said to be a \emph{left order} if it is invariant under left multiplication, that is, $x<y$ implies $zx<zy$ for all $x,y,z\in G$. Similarly, a total relation $<$ is a \emph{right order} if it is invariant under right multiplication. A relation that is simultaneously a left and a right order is called a \emph{bi-order}. 
	A group that admits a left order is called \emph{left-orderable}, and a group that admits a bi-order is called \emph{bi-orderable}.
	
	Elements that are bigger or smaller than the identity element of a group are called \textbf{positive} and \textbf{negative} respectively.
	The set $P_<$ of all positive elements of a group $G$ with order $<$ is called a \emph{positive cone}. The positive cone $P_<$ is usually identified with the ordering $<$.

	Given a group $G$, we denote by $\LO(G)$ (resp. $\BO(G)$) the set of all left-orderings (resp. bi-orderings) on $G$. 
	For a finite subset $F \subset G$, let $V_F$ (resp. $U_F$) be a set of all left-orders (resp. bi-orders) $<$ such that $F\subset P_<$. These sets form the base of the natural topology on $\LO(G)$ ($\BO(G)$) as first explored by Sikora in \cite{Sikora}.
	In the case of a countable group $G$, he proved that $\LO(G)$ is a compact totally disconnected Hausdorff metrizable space. Then $\BO(G)$ is a closed subset of $LO(G)$.
	
	The ordering $<$ is isolated in this topology, if there is a finite set $F\subset G$ such that $<$ is the only order with the property $F\subset P_<$. If a countable ordered group admits no isolated orderings then $\LO(G)$ ($\BO(G)$) is isomorphic to a Cantor set.
	
	Thus, asking whether a given group has an isolated ordering is a natural question in the theory of orderable groups. 
	In the case of free abelian groups $\Z^n$, $n\geq 2$, Sikora showed that $\Z^n$ does not admit isolated orderings.
	For free groups $F_n$, $n\geq 2$, McCleary proved is \cite{McCLeary} the absence of isolated left-orderings. A similar result was later shown for bi-orders in \cite{my1}. McCleary's result was generalized in \cite{Rivas} for free products of left-ordered groups.
	In this paper, we prove a bi-ordered version of this fact. 
	
	\begin{theorem}\label{freenoiso}
		Let $G$ and $H$ be two bi-orderable groups. Then the free product $G\bigast H$ has no isolated bi-orderings.
	\end{theorem}
	
	Our method relies on the dynamical realization of bi-ordered groups introduced in \cite{my1}. The standard dynamical realization of a left-ordered group $G$ associates the elements of $G$ with rational numbers and translates the left multiplication action of $G$ on itself to an action on $\Q$ which later extends to an action on $\R$. The alternative dynamical realization of a bi-ordered group $G$ associates convex jumps in $G$ with some disjoint intervals in $\R$ and translates the conjugation action of $G$ on the set of convex jumps to an action on the set of endpoints of those intervals. The action of an element $g\in G$ inside the interval $I$ is defined analogously to the action of a certain element $s_{g,I}$ in the dynamical realization of some left-ordered group $S$.
	The specific details of this construction and its relevant properties are given in Section \ref{sec:2}.
	To illustrate the strength of this dynamical realization, we reprove Vinogradov's theorem \cite{Vin}.
	\begin{theorem}
		The free product of two bi-ordered groups is bi-ordered.
	\end{theorem}
	Of course, our method is limited to countable groups, however, one can easily extend to the general case using the compactness argument.

	In \cite{clay} Clay presented an alternative proof of the fact that the free group $F_2$ does not admit isolated left orders. He showed that there exists a dense orbit under the natural conjugation action of $F_2$ on $\LO(F_2)$. 
	Of course, $F_2$ acts trivially on $\BO(F_2)$ by conjugation; thus, this action does not have dense orbits. However, it is reasonable to ask whether there exists a dense order under the natural action of the automorphism group $\Aut_{F_2}$ on $\BO(F_2)$. For instance, under the automorphism action on $\BO(\Z^2)$, all orbits are dense.
 
	We give a negative answer to this question.
	\begin{theorem}\label{nodense}
		Let $F_2$ be a free group on two generators. Then the natural action of $\Aut_{F_2}$ on $\BO(F_2)$ has no dense orbits.
	\end{theorem}
	
	Moreover, we prove a more general result.
	
	\begin{theorem}\label{verynodense}
		The set $\BO(F_2)$ cannot be presented in the form
		\[
		\BO(F_2)=\bigcup_{n\in \N} \overline{\Aut_{F_2}(<_n)}
		\]
		where $\overline{\Aut_{F_2}(<_n)}$ denotes the closure of the orbit of order $<_n\in \BO(F_2)$.
	\end{theorem}

	\section{Dynamical Realization}
	
	\label{sec:2}
	
	The connection between left-orderability and the group's dynamical properties is a widely recognized fact. An illustration of this connection can be found in the theorem below, as presented in \cite[Theorem 3.4.1]{KopMed}.
	\begin{theorem}\label{dynrel}
		A group $G$ is left-orderable if and only if it acts faithfully on some totally ordered set $X$ by order-preserving homeomorphisms. 
		For countable $G$ one can take $X=\Q$ or $X=\R$.
	\end{theorem}
	
	The action constructed to prove Theorem \ref{dynrel} is called the \emph{dynamical realization} of the (left-ordered) group $G$.

	A similar result holds for bi-orderable groups.
	
	\begin{theorem}\label{dynre}
		A countable bi-ordered group $G$ admits a faithful action by orientation-preserving homeomorphisms of the real line. Moreover, the action can be chosen to satisfy the conditions 
		\begin{equation}\label{dynbi}
			\sup\{x\in \R\mid g(x)>x\}>\sup\{x\in\R \mid g(x)<x\}, \quad g>1,
		\end{equation}
		and
		\begin{equation}\label{dynnol}
			\sup\{x\in \R\mid g(x)\neq x\}>\sup\{x\in\R \mid h(x)\neq x\}, \quad g\gg h.
		\end{equation}
	\end{theorem}

	\begin{remark}
		The standard dynamical realization of $G$ as a left-ordered group satisfies \eqref{dynbi} but does not satisfy \eqref{dynnol}.
	\end{remark}
	
	The proof of Theorem \ref{dynre} can be found in \cite{my1}. However, the proof there relies on the bi-orderability of free products. Since we aim to reprove this fact, we present an alternative proof of Theorem \ref{dynre} in this paper.

	\begin{proof}[Proof of Theorem \ref{dynre}]
		
		Let $G$ be a countable bi-ordered group, and let $\J$ be the set of convex jumps in $G$. The group $G$ acts on $J$ by conjugations splitting it into orbits
		\[
		\J=\bigcup \0_i.
		\]
		For each orbit we fix a representative $J_i=(C_i<D_i)\in \0_i$, and associate it with a group $H_i\coloneqq  ^{N_i}\!/\!_{C_i}$, where $N_i=N_{C_i}=N_{D_i}$ is the normalizer of the jump $J_i$ in $G$. For each $J\in\0_i$ we also choose an element $g_J\in G$ such that $J^{g_J}=J_i$. Then, for $g\in G$ and $J\in\J$ we denote 
		\[
		h_{g,J}\coloneqq  \left( g_{J^g} \cdot g \cdot g^{-1}_J\right) C_i \in H_i.
		\]
		As in \cite{my1} we let 
		\[
		S=\left(\Asterisk_i H_i \right)\Asterisk F_{\J},
		\]
		where $F_{\J}=\langle f_J|J\in\J \rangle$ is a free group whose generators are indexed with convex jumps of $G$.
		Then $S$ is a free product of bi-ordered groups. So, it is left-ordered using the argument \cite{DZtrees}, and the order on $S$ extends the orders on its factors. Moreover, for each positive $h\in H_i$, and any $s\in S$ we have $h^s>0$ (the order on $S$, however, does not need to be bi-invariant).
  
		We set
		\begin{equation}\label{s}
			s_{g,J}\coloneqq  f_{J^g}\left( h_{g,J} J_i\right) f^{-1}_{J},\quad (g,J)\in G\times \0_i
		\end{equation}
		and observe that they satisfy the conditions
		\begin{enumerate}
			\item \label{c1} $s_{g,J}=1$ if and only if $g\in C$, where $J=(C<D)$;
			\item \label{c2} $s_{g,J_g}>1$ for all $g>1$ and $s_{g,J_g}<1$ for all $g<1$;
			\item \label{c3}  $s_{g,J^h}\cdot s_{h,J}=s_{gh,J}$, for all $g,h\in G$, $J\in\J$.
		\end{enumerate}
		
		Next, we define the action of $G$ on $\J\times S$ by
		\begin{equation}\label{act}
			g(J,s)=(J^g, s_{g,J}\cdot s).    
		\end{equation}
		
		Note that condition \ref{c3} implies that \eqref{act} defines a group action. 
		
		We equip $\J\times S$ with lexicographic order so that it is countable dense and unbounded. Then the action given by \eqref{act} preserves this order because the conjugation action on $\J$ and the left multiplication action on $S$ are order-preserving.
		
		Consider some $g\in G\setminus \{1\}$. Then 
		\[
		g(J,s)=(J,s)
		\]
		if and only if $J>J_g$, and, if $g>0$
		\[
		g(J_g,s)=(J_g,s_{g,J_g}s)>(J,s)
		\]
		for all $s\in S$.
		
		By Cantor's isomorphism theorem, there is an order-preserving bijection $\sigma:\J\times S\to \Q$. We build the action on $\Q$ by the rule
		\[
		g(\sigma(J,s))= \sigma (g(J,s)).
		\]
		
		This action extends continuously to the action on $\R$, which satisfies conditions \eqref{dynbi} and \eqref{dynnol}.
	\end{proof}

	In the future, we will call the \emph{dynamical realization} of a countable bi-ordered group $G$ any its action on $\R$ that satisfies \eqref{dynbi}, \eqref{dynnol}, and such that the action of any $g\neq 1$ is non-trivial on any interval $(a,b)$ with $b<r_g\coloneqq \sup\{x\in\R\mid g(x)\neq x \}$.
	We call this point $r_g$ the \emph{critical point} for the element $g\in G$ and denote the set of all critical points $T_G\coloneqq \{t_g\mid g\in G\setminus\{1\}\}$. It is possible to have $t_g=+\infty$ for some $g\in G$. To avoid this, we pass to the dynamical realization of $\Z\times G$ with lexicographical order.
	We associate the dynamical realization of $G$ with the embedding $\rho:G\to \h$.

	Let the dynamical realizations $\rho_G$ and $\rho_H$ of bi-ordered groups $G$ and $H$ be given. The actions of $G$ and $H$ generate the action of $F=G\ast H$ on $\R$.
	
	Let $f=f_n \dots f_2 f_1$ be an element of $F$ written in the normal form, i.e., each $f_i$ is a non-trivial element of either $G$ or $H$, and the consecutive elements are from different groups. 
	For $x\in \R$ we say $f$ is $x$-reduced if 
	\[
	t_{f_k}\geq f_{k-1}\dots f_2 f_1 (x), \; 1\leq k\leq n.
	\]
	If $f$ is not $x$-reduced, we delete all letters $f_k$ with $t_{f_k}< f_{k-1}\dots f_1 (x)$, put the resulting word in the normal form, and repeat this process until it terminates at some $x$-reduced word. We call this word the $x$-reduced version of $f$ and denote it by $f_x$. 
	Note that $f$ and $f_x$ acts identically on $[x,\infty)$ and have the same left germs at $x$.

	We say that the dynamical realizations $\rho_G$ and $\rho_H$ of $G$ and $H$ respectively \emph{merge} if for any $x\in T_G \bigcap T_H$, and any $x$-reduced $f\in F=G\ast H$, $f(x)=x$ is possible only when $x\in T_G$ and $f\in G$, or $x\in T_H$ and $f\in H$.

	\begin{proposition}\label{prop}
		Let the dynamical realizations $\rho_G:G\to \h$ and $\rho_H:H\to \h$ be given. Then there is a homeomorphism $\phi\in\h$ such that $\rho_G$ and $\rho^{\prime}_H=\phi\circ \rho_H\circ\phi^{-1}$ merge.
		Moreover, for any $\epsilon>0$ one can take $\phi$ to satisfy
		\[
		\|\phi\|\coloneqq\sup\{|\phi(x)-x|, x\in \R\}<\epsilon.
		\]
	\end{proposition}
	
	\begin{proof}
		Order all pairs $(f,f_0)$ with $f\in F$ and $f_0\in G\bigcup H\setminus\{1\}$ such that the pair $(f^{\prime},f_0)$ comes before $(f,f_0)$ if the word $f^{\prime}$ can be obtained from $f$ by canceling some of its letter. We denote the $k$th pair $(f^{(k)},f_0^{(k)})$ and let $t_k=t_{f_0^{(k)}}$ be the critical point for $f_0^{(k)}$. If $f_0^{(k)}\in H$, we take $t_k$ related to the initial dynamical realization $\rho_H$. Let also $t_k^{\prime}=t_k$, when $f_0^{(k)}\in G$ and $t_k^{\prime}=\phi(t_k)$, when $f_0^{(k)}\in H$.

		We will build the homeomorphism $\phi$ recursively in the form
		\[
		\phi=\phi_1 \phi_2 \dots,
		\]
		where $\phi_k$ ensures that the pair $(f^{(k)},t^{\prime}_k)$ satisfies the condition in the definition of merging. In this case, we call the pair $(f^{(k)},t^{\prime}_k)$ good.
		
		We will use the notation $\widetilde{t}_k=t_k$ when $f_0^{(k)}\in G$ and $\widetilde{t}_k=\phi_1\dots\phi_k(t_k)$ when $f_0^{(k)}\in H$.
		
		We will require $\|\phi_k\|<\epsilon/2^k$ to guarantee $\|\phi\|<\epsilon$.
		
		We begin the recursion by choosing an arbitrary $\phi_1$ that turns $(f^{(1)},\tilde{t}_1)$ into a good pair and satisfies $\|\phi_1\|<\epsilon/2$.
		Due to continuity, $f^{(1)}_{t^{\prime}_1}$ remains the same word, and the pair $(f^{(1)},t^{\prime}_1$ remains good if the tail $\phi_2 \phi_3 \dots$ is small enough. Namely, $\|\phi_2 \phi_3 \dots\|<\epsilon_1$ for some $\epsilon_1>0$. We achieve this by requiring $\|\phi_k\|<\epsilon_1/2^{k-1}$ for $k>1$.

		Similarly, in the $k$-th step, we choose the homeomorpism $\phi_k$ that turns $(f^{(k)},\widetilde{t}_k)$ into a good pair and satisfies $\|\phi_k\|<\min\{\epsilon/2^k,\epsilon_1/2^{k-1},\dots,\epsilon_{k-1}/2\}$. Again, due to continuity, $f^{(k)}_{t^{\prime}_k}$ remains the same word, and the pair $(f^{(k)},t^{\prime}_k$ remains good if the tail $\|\phi_{k+1} \phi_{k+2} \dots\|<\epsilon_k$ for some $\epsilon_k>0$. So we require $\|\phi_m\|<\epsilon_1/2^{m-k}$ for $m>k$ in addition to the already existing constraints.
		
		Clearly, the process described above produces $\phi$ as desired.
		
	\end{proof}

	\begin{theorem}\label{merging}
		Assume that the dynamical realizations $\rho_G:G\to \h$ and $\rho_H:H\to \h$ merge. Then the generated action $\rho_F$ satisfies the condition \eqref{dynbi}. Furthermore, the associated order on $F$ extends the orders on $G$ and $H$.
	\end{theorem}
	
	\begin{proof}
		First, note that the sets $T_G$ and $T_H$ are disjoint. Indeed, if $t_g=t_h=t$ for some $g\in G$, $h\in H$, then the pair $(f=hgh,t)$ violates the condition of merging as $f_t=f$ and $f(t)=t$.
		Furthermore, the orbits of $T_G$ and $T_H$ are disjoint. Indeed, let $f_1(t_g)=f_2(t_h)$ for some $f_1,f_2\in F$, $g\in G$ and $h\in H$. Then $hf_2^{-1}f_1(t_g)=h(t_h)=t_h$ and $f_1^{-1}f_2hf_2^{-1}f_1(t_g)=t_g$ but the middle letter cannot be deleted during the $t_h$ reduction process of the word $f_1^{-1}f_2hf_2^{-1}f_1=h^{f_1^{-1}f_2}$.
		Similarly, if the orbits of $t_{g_1}$ and $t_{g_2}$ are disjoint under the action of $G$, they remain disjoint under the action of $F$.
		
		Let $f=f_n \dots f_2 f_1\neq 1$ be given. Let $f^{\prime}=f_{k_1}\dots f_{k_i}\neq 1$ be a word obtained from $f$ by deleting some of its letters so that all the letters $f_{k_1},\dots, f_{k_i}$ belong to the same group. We denote $t^{\prime}=(f_{k_i-1}\dots f_1)^{-1} (t_{f^{\prime}})$. The element $f^{\prime}$ may appear as a single letter during the $x$-reduction process of the word $f$. Thus, the form of the $x$-reduced word $f_x$ may depend on whether $x\leq t^{\prime}$ or $x< t^{\prime}$.
		Let $t_1>\dots>t_m$ be all points of the form $t^{\prime}$ written in descending order. Then the word $f_x$ is constant for $x\in(t_{k+1},t_k]$, $1\leq k<m$. Also, $f_x=f$ for $x\leq t_m$ and $f_x=1$ for $x>t_1$.
		
		Let $x$ be the largest number with $f_x\neq 1$. In particular, $x=t_r$ is the largest among $t_1,\dots, t_m$ with this property. We write 
		\[
		f_x=\widehat{f}_{l}\widehat{f}_{l-1}\dots \widehat{f}_2 \widehat{f}_1
		\]
		in the normal form.
		
		Since $f_y\neq f_x$ for $y>x$ there is a letter $\widehat{f}_i$ with $t_{\hat{f}_i}=\widehat{f}_{i-1}\dots \widehat{f}_2 \widehat{f}_1(x)$. If there is another such letter $\widehat{f}_j$ then 
		\[
		t_{\widehat{f}_j}=\widehat{f}_{j-1}\dots \widehat{f}_{i+1}(t_{\widehat{f}_i})
		\]
		assuming that $j>i$. But this is impossible in the reduced word. So $\widehat{f}_i$ is unique, and we write 
		\[
		f_x=w_1 \widehat{f}_i w_0=w_1 \widehat{f}_i w_1^{-1} w_1 w_0=\left(\widehat{f}_i\right)^{w_1} w,
		\]
		where $w_0=\widehat{f}_{i-1}\dots \widehat{f}_1$, $w_1=\widehat{f}_{l}\dots \widehat{f}_{i+1}$ and $w=w_1w_0$.
		
		Since $\widehat{f}_i$ is the only letter that satisfies $t_{\hat{f}_i}=\widehat{f}_{i-1}\dots \widehat{f}_2 \widehat{f}_1(x)$, $t_r$ does not appear as $t^{\prime}$ for $w$ and $w_y$ is constant on $(t_{r+1},t_{r-1}]$. But $w_y=f_y=1$ for $y>x=t_r$. So $w_y=1$ for some $y<x$ and $w$ acts trivially in the neighborhood of $x$. Then, the action of $f$ near $x$ is equivalent to the action of $\left(\widehat{f}_i\right)^{w_1}$ and we have $f>1$ if $\widehat{f}_i>1$. The condition \eqref{dynbi} now follows from the fact that it holds for $\widehat{f}_i$.
		
	\end{proof}

	\begin{remark}
		Theorem \ref{merging} gives a new proof for Vinogradov's theorem or bi-orderability of free products of bi-ordered groups.
	\end{remark}

	\section{Free Products}

	In this section, we show that no ordering on $F=G\bigast H$ is isolated.
	
	That means, for any order $<$ on $F$ and any finite collection of positive elements 
	\begin{equation*}
		1<f_1<\dots<f_n,
	\end{equation*}we need to construct an order $\prec$, different from $<$ that still satisfy 
	\begin{equation}\label{need}
		1\prec f_1, \dots, f_n.
	\end{equation}

	Our strategy is similar to the one used in \cite{Rivas} to show that no left ordering on the free product of left ordered groups is isolated. We are going to conjugate one of the factors of $F=G\bigast H$ by some $\tau\in\h$. We will perform this conjugation in the merging of the dynamical realizations of $F$ and $\Z$.

	Let $C_{f_1}$ be the largest convex subgroup (with respect to the order $<$) of $F$ that does not contain $f_1$. Then, since $1<f_1<f_2<\dots<f_n$, $C_{f_1}$ does not contain any of $f_i$'s. Consider $\Gamma\coloneqq \bigcap_{f\in F}\left(C_{f_1}\right)^f$. Clearly, $\Gamma$ is convex and normal in the $F$ group, which does not contain any of the $f_i$'s. If $\Gamma\neq\{1\}$ we reverse the order on it. To be precise, the new order $\prec$ is defined by
	\begin{equation*}
		f\succ 1 \iff f>1 \text{ and } f\notin \Gamma \text{ or }  f<1 \text{ and } f\in \Gamma.
	\end{equation*}
	It is straightforward to check that $\prec$ is a bi-ordering of $F$ that satisfies \eqref{need}. Also, $\prec$ is different from $<$ under the assumption that $\Gamma$ is non-trivial.
	
	Consider now the case $\Gamma=\{1\}$.
	
	This implies that for any $f_0\in F\setminus\{1\}$ there is $f\in F$ such that 
	\begin{equation}\label{small}
		f_1^f\ll f_0.
	\end{equation}

	Let $f_0$ be the smallest of the elements of the form $f_1^u$, where $u$ is a subword of some $f_i$.
	Then, by \eqref{small}, there exists $f^{\prime}\in F$ such that
	\begin{equation*}
		f_1^{f^{\prime}}\ll f_0,
	\end{equation*}
	and, because the dynamical realization of $\widetilde{F}$ satisfies \eqref{dynnol} we have
	\begin{equation*}
		f^{\prime}(t_{f_1})=t_{f_1^{f^{\prime}}}<t_{f_0}.
	\end{equation*}
	To simplify the notation, we denote $t_1=t_{f_1}$ and $t_0=t_{f_0}$.

	We write $f^{\prime}=f_m^{\prime}\dots f^{\prime}_1$ in the normal form. Let $t_1^{\prime}=f^{\prime}_1(t_1)$,  $t_2^{\prime}=f^{\prime}_2(t_1^{\prime})$, \dots, $t_m^{\prime}=f^{\prime}_m(t_{m-1}^{\prime})=f^{\prime}(t_1)$. We may assume that $t_1>t_1^{\prime}>t_2^{\prime}>\dots>t_{m}^{\prime}$. Otherwise, we delete the elements $f^{\prime}_i$ in $f$ that do not decrease the correspondent $t^{\prime}_i$. In addition, we assume that $t_m^{\prime}$ is the first one that is smaller than $t_0$, so that $t_{m-1}^{\prime}\geq t_0>t_m^{\prime}$. Otherwise, we remove the tail of $f^{\prime}$.
	Let also $t_{m+k}^{\prime}=\left(f_m^{\prime}\right)^k(t_m^{\prime})$, $k\geq 1$.

	Next, we construct a dynamical realization $\rho_{\Z}$ of $\Z=\langle \tau \rangle$.

	Without loss of generality, let $f^{\prime}_m\in H$. Let $g\in G$ be any non-trivial element, and let $t^{\prime\prime}<t_m^{\prime}$ be such that 
	\[
	t^{\prime\prime\prime}=g(t^{\prime\prime})< t^{\prime\prime}.
	\]
	If such a point does not exist, we replace $g$ with $g^{-1}$.
	
	Pick any $t^{\prime}\in \R$ such that $t_m^{\prime}<t^{\prime}<t_0$. 
	We let $\tau_1$ act on $\R$ such that
	\begin{enumerate}
		\item $\tau_1$ acts trivially on $[t^{\prime},\infty)$;
		\item $\tau_1(t^{\prime\prime})=t_m^{\prime}$;
		\item $\tau_1(t^{\prime\prime\prime})=(t_{m+1}^{\prime})$.
	\end{enumerate}
	This is possible because
	\begin{equation*}
		t^{\prime\prime\prime}<t^{\prime\prime}<t^{\prime}
	\end{equation*}
	and
	\begin{equation*}
		t_{m+1}^{\prime}<t_m^{\prime}<t^{\prime}.
	\end{equation*}
	
	For example, we may set
	\begin{equation*}
		\tau_1(x)=\begin{cases}
			x, & \quad t^{\prime}<x\\
			\frac{t^{\prime}-t_m^{\prime}}{t^{\prime}-t^{\prime\prime}}(x-t^{\prime})+t^{\prime}, & \quad    t^{\prime\prime}< x \leq t^{\prime}\\
			\frac{t_m^{\prime}-t_{m+1}^{\prime}}{t^{\prime\prime}-t^{\prime\prime\prime}}(x-t^{\prime\prime})+t^{\prime}_m  , & \quad   x \leq t^{\prime\prime}\\
		\end{cases}
	\end{equation*}
	
	Note that
	\begin{equation}\label{tau1}
		\tau_1 g \tau_1^{-1} (t_m^{\prime})=\tau_1 g (t^{\prime\prime})=\tau_1(t^{\prime\prime\prime})=t^{\prime}_{m+1}>t^{\prime}_{m+2}=f_{m}^{\prime}(t^{\prime}_{m+1}).
	\end{equation}
	
	The dynamical realization of $\Z$ generated by $\tau_1$ does not merge with the dynamical realization $\rho_F$ of $F$. However, by Proposition \ref{prop}, there is a homeomorphism $\phi_1$ such that the dynamical realization $\rho_1$ of $\Z$ generated by $\tau_1^{\prime}=\tau_1^{\phi_1}$ does merge with $\rho_F$. We choose $\phi_1$ to be sufficiently small so that $\tau^{\prime}_1 g \left(\tau^{\prime}_1\right)^{-1}>t^{\prime}_{m+2}$, and $t_{\tau^{\prime}_1}<t_0$. This is possible because $\tau_1$ satisfies the inequality $\eqref{tau1}$ and $t_{\tau_1}=t^{\prime}<t_0$.

	Now, we can consider the merging of the dynamical realizations $\rho_F$ and $\rho_1$. It induces the ordering of $F\bigast\Z$. For $f=g_1h_1\dots g_kh_k\in F$, where $g_i\in G$, $h_i\in H$ denote 
	\begin{equation*}
		(f)_{1}\coloneqq g_1^{\tau_1^{\prime}}h_1\dots g^{\tau_1^{\prime}}_kh_k.
	\end{equation*}
	Let the order $\prec_1$ be given by $f\succ_1 1$ whenever $(f)_{1}>1$.
	
	Let $f_i=g^{(i)}_1h_1^{(i)}\dots g^{(i)}_{k_i}h_{k_i}^{(i)}$ and consider
	\begin{align*}
		(f_i)_{\tau_1^{\prime}}&=\left(g^{(i)}_1\right)^{\tau_1^{\prime}}h_1^{(i)}\dots \left(g^{(i)}_{k_i}\right)^{\tau_1^{\prime}}h_{k_i}^{(i)} \\
		&=\tau_1^{\prime} g^{(i)}_1 \left(\tau_1^{\prime}\right)^{-1} h_1^{(i)}\dots \tau_1^{\prime} g^{(i)}_{k_i} \left(\tau_1^{\prime}\right)^{-1} h_{k_i}^{(i)}\\
		&=f_i   \left(\tau_1^{\prime}\right)^{\left(g_1^{(i)}h_1^{(i)}\dots g_{k_i}^{(i)}h_{k_i}^{(i)}\right)^{-1}} \left(\left(\tau_1^{\prime}\right)^{-1}\right)^{\left(g_1^{(i)}h_1^{(i)}\dots g_{k_i}^{(i)}h_{k_i}^{(i)}\right)^{-1}} \dots \left(\left(\tau_1^{\prime}\right)^{-1}\right)^{\left(h_{k_i}^{(i)}\right)^{-1}}   \\
		&=f_i \prod \left(\left(\tau_1^{\prime}\right)^{\pm 1}\right)^{u^{-1}},
	\end{align*}
	where $u$ runs through some subwords of $f_i$.

	Recall that $t_{\tau_1^{\prime}}<t_0=f_{t_0}$, and, therefore $\tau_1^{\prime}\ll f_0$. Thus, since $f_0$ is defined as the smallest of $f_1^u$'s,
	\begin{equation*}
		\left(\tau_1^{\prime}\right)^{u^{-1}}\ll (f_0)^{u^{-1}}\leq \left(f_1^{u}\right)^{u^{-1}}=f_1\leq f_i.
	\end{equation*}
	Therefore, $(f_i)_{1}>1$ and $f_i\succ_1 1$. So, \eqref{need} is satisfied.
	
	For the same reason, $t_{(f_i)_1}=t_{f_i}$. In particular, $t_{(f_1)_1}=t_{f_1}=f_1$.
	Then, $(f^{\prime}_1)_1(t_1)=f^{\prime}_1(t_1)=t_1^{\prime}$,  $(f^{\prime}_2)_1(t^{\prime}_1)=f^{\prime}_1(t^{\prime}_1)=t_2^{\prime}$, $\dots$, $(f^{\prime}_{m-1})_1(t^{\prime}_{m-2})=f^{\prime}_{m-1}(t^{\prime}_{m-2})=t_{m-1}^{\prime}$ because $\tau$ acts trivially on $t^{\prime}_i$'s, $i=1,\dots, m-1$ as $t_i>t_1$. Also, $(f^{\prime}_m)_1(t^{\prime}_{m-1})=f^{\prime}_m(t_{m-1}^{\prime})=t_m^{\prime}$ since $f_m^{\prime}\in H$ is not affected by the conjugation by $\tau$. Therefore,
	\begin{equation*}
		t_{\left(f_1^{f^{\prime}}\right)_1}=t_{f_1^{f^{\prime}}}=t_m^{\prime}
	\end{equation*}
	and
	\begin{equation*}
		t_{\left(f_1^{gf^{\prime}}\right)_1}=(g)_1(t_m^{\prime})=g^{\tau^{\prime}_1}(t_m^{\prime})>t_{m+2}^{\prime}=t_{f_1^{f_m^{\prime}f^{\prime}}}
	\end{equation*}
	so
	\begin{equation*}
		f_1^{gf^{\prime}}\succ_1 f_1^{f_m^{\prime}f^{\prime}}
	\end{equation*}

	Similarly, we let $\tau_2$ act on $\R$ such that
	\begin{enumerate}
		\item $\tau_2$ acts trivially on $[t^{\prime},\infty)$;
		\item $\tau_2(t^{\prime\prime})=t_m^{\prime}$;
		\item $\tau_2(t^{\prime\prime\prime})=(t_{m+3}^{\prime})$
	\end{enumerate}
	so that
	\begin{equation}\label{tau2}
		\tau_2 g \tau_2^{-1} (t_m^{\prime})=\tau_2 g (t^{\prime\prime})=\tau_2(t^{\prime\prime\prime})=t^{\prime}_{m+3}<t^{\prime}_{m+2}=f_{m}^{\prime}(t^{\prime}_{m+1}).
	\end{equation}
	
	We choose the adjustment $\tau_2^{\prime}$ of $\tau_2$ that merges with $\rho_F$ to satisfy the inequality \eqref{tau2} and $t_{\tau^{\prime}_2}<t_0$. Then we construct the order $\prec_2$ analogously to the order $\prec_1$. Again, we have $f_i\succ_2 1$, $1\leq i\leq n$. But now
	\begin{equation*}
		t_{\left(f_1^{gf^{\prime}}\right)_2}=(g)_1(t_m^{\prime})=g^{\tau^{\prime}_2}(t_m^{\prime})<t_{m+2}^{\prime}=t_{f_1^{f_m^{\prime}f^{\prime}}}
	\end{equation*}
	and
	\begin{equation*}
		f_1^{gf^{\prime}}\prec_2 f_1^{f_m^{\prime}f^{\prime}}.
	\end{equation*}
	
	Therefore, the orders $\prec_1$ and $\prec_2$ are different, and at least one of them differs from the initial order $<$. We choose $\prec$ to be this order to conclude the proof.

	\section{Absence of dense orbits}

	The following proposition follows from the work of Conrad \cite{Conrad}, due to the fact that all bi-orders are Conradian.
	
	\begin{proposition}
		For every convex jump $C<D$ in a bi-ordered group $G$ there is homomorphism $\phi: D\to \R$ such that $\ker \phi=C$ and $\phi(g)>0$ for all positive $g\in D\setminus C$. This homomorphism is defined uniquely up to a positive multiple.
	\end{proposition}
	
	The homomorphism $\phi$ above is known as the \emph{Conrad homomorphism} associated with the jump $C<D$.
	
	The normalizer group $N=N_C=N_D$ acts on $D$ by conjugations. The conjugation by $g\in N$ combined with the Conradian homomorphism $\phi$ results in an automorphism of $\phi (D)<\R$, given by $\phi(h)\mapsto \phi(h^g)$. This automorphism preserves the natural order on $\phi(D)$, and therefore is a multiplication by some $\alpha_g>0$.

	Let the order $<$ on $F_2$ be given and assume that there is a pair of positive generators $a,b$ such that $C_a=D_b$ and the conjugation by $a$ on $D_b$ corresponds to the multiplication by $\alpha>0$ in $\R$. If such a pair of generators exists, we say that the order is of type $\alpha$.
	
	The type is not determined for all orders. For example, every order generated by merging dynamical realizations of $\langle a\rangle$ and $\langle b \rangle$ is typeless as in it $\left(D_b\right)^a\neq D_b$ (assuming $a\gg b$). However, if the order has a type, it is determined uniquely.

	\begin{lemma}
		Let the order $<$ be of type $\alpha$ for some $\alpha>0$. Then it is not of the type $\beta\neq \alpha$.
	\end{lemma}
	\begin{proof}
		Since the order is of type $\alpha$ there are generators $a,b>1$ of $F_2$ such that $C_a=D_b$, and $a$ acts on $\faktor{D_b}{C_b}$ as multiplication by $\alpha$.
		
		Under these conditions, $D_a=F_2$ is the largest convex subgroup of $F_2$, and $C_a=D_b$ is the second largest convex subgroup of $F_2$.
		We define the homomorphisms $\phi:F_2\to \R$ and $\psi:C_a\to \R$ by $\phi(a)=1$, $\phi(b)=0$, and $\psi(b^{a^n})=\alpha^n$ so that $\ker \phi=C_a$ and $\ker \psi= C_b$.
		
		Clearly, $\phi$ is the Conrad homomorphism for the jump $C_a<D_a$. Additionally, the homomorphism $\psi$ satisfies $\psi(g^a)=\alpha \psi(g)$, therefore, it is the Conrad homomorphism for the jump $C_b<D_b$.
		
		Let $c,d>1$ be any other pair of generators of $F_2$ satisfying $C_c=D_d$.
		Then $D_c=F_2$ is the largest convex subgroup of $F_2$ and $C_c=D_d$ is the second largest convex subgroup of $F_2$. Therefore, $D_c=D_a$, $C_c=C_a$, and $C_d=C_b$. Since $d\in D_d=C_a$, we have $\phi(d)=0$. So $\phi(F_2)=\Z$ is generated by $\phi(c)$, which means that $\phi(c)=\pm 1$. We rule out $\phi(c)=-1$ because $c>1$, so $\phi(c)=1=\phi(a)$. We write
		\[
		c=(ca^{-1})a
		\]
		and observe that $ca^{-1}\in \ker\phi=C_a$. Then, for any $g\in C_a$, we have
		\[
		\psi(g^c)=\psi(g^{(ca^{-1})a})=\psi(ca^{-1})+\psi(g^a)-\psi(ca^{-1})=\alpha\psi(g).
		\]
		
		In other words, conjugation by $c$ corresponds to multiplication by $\alpha$, and the order type does not depend on the choice of generators of $F_2$. 
	\end{proof}

	\begin{proof}[Proof of Theorem \ref{verynodense}]
		Let $<_{\alpha}$ be an order of type $\alpha>1$. Let $k,l,m,n$ be positive integers such that
		\[
		1<\frac{k}{l}<\alpha<\frac{m}{n}
		\]
		Then the following inequalities are satisfied:
		\begin{itemize}
			\item $1 <_{\alpha} b <_{\alpha} a$;
			\item $\left(b^a\right)^n <_{\alpha} b^m$;
			\item $b^l <_{\alpha} \left(b^a\right)^k$.
		\end{itemize}

		Let $V_{k,l,m,n}\subset \BO(F_2)$ be the set of all orders satisfying these inequalities and let
		\[
		U_{k,l,m,n}=\bigcup_{\sigma\in \Aut_{F_2}} \sigma(V_{k,l,m,n}).
		\]
		Then the sets $V_{k,l,m,n}$ are open as elements of the base of the topology on $\Aut_{F_2}$, and the sets $U_{k,l,m,n}$ are open as unions of open sets. Also, note that the sets $U_{k,l,m,n}$ are constructed to be invariant under the action of $\Aut_{F_2}$
		
		Furthermore, every order in $V_{k,l,m,n}$ (hence, in $U_{k,l,m,n}$) is of type $\beta$ for some $\beta\in [k/l,m/n]$.
		So, whenever $1<{k_1}/{l_1}<{m_1}/{n_1}<{k_2}/{l_2}<{m_2}/{n_2}$, the sets $U_{k_1,l_1,m_1,n_1}$ and $U_{k_2,l_2,m_2,n_2}$ are disjoint.
		
		Let $<_{\alpha}$ and $<_{\beta}$ be orders of types $\alpha$ and $\beta$, respectively, and assume that $1<\alpha<\beta$. Choose numbers $k_1,l_1,m_1,n_1,k_2,l_2,m_2,n_2\in \N$ such that
		\[
		1<\frac{k_1}{l_1}<\alpha<\frac{m_1}{n_1}<\frac{k_2}{l_2}<\beta<\frac{m_2}{n_2}.
		\]
		
		Then $<_{\alpha}\in U_{k_1,l_1,m_1,n_1}$, $<_\beta\in U_{k_2,l_2,m_2,n_2}$, and $U_{k_1,l_1,m_1,n_1}\cap U_{k_2,l_2,m_2,n_2}=\emptyset$. 
		
		Let $\prec$ be any order on $F_2$. Then $\prec\notin U_{k_1,l_1,m_1,n_1}$ or $\prec\notin U_{k_2,l_2,m_2,n_2}$. If $\prec\notin U_{k_1,l_1,m_1,n_1}$, then, since $U_{k_1,l_1,m_1,n_1}$ is invariant, $\Aut_{F_2}(\prec)\cap _{k_1,l_1,m_1,n_1}=\emptyset$; and, since $U_{k_1,l_1,m_1,n_1}$ is open, $\overline{\Aut_{F_2}(\prec)}\cap U_{k_1,l_1,m_1,n_1}=\emptyset$. In this case, $<_{\alpha}\notin \overline{\Aut_{F_2}(\prec)}$. Similarly, if $\prec\notin U_{k_2,l_2,m_2,n_2}$, then $<_{\beta}\notin \overline{\Aut_{F_2}(\prec)}$. Therefore, the set $\overline{\Aut_{F_2}(\prec)}$ may contain orders of at most one type. It remains to show that for every $\alpha>1$ there is an order of that type.
		
		To construct an order of type $\alpha>1$ we begin with choosing an arbitrary order $<\in \BO(F_2)$ and defining the homeomorphisms $\phi:F_2\to \R$, ${\psi:\langle \left(b^a\right)^n\mid n\in \Z \rangle\to \R}$ as above. Next, we define the order $<_{\alpha}$ by declaring $g >_{\alpha}$ if
		\begin{itemize}
			\item $\phi(g)>0$, or
			\item $g\in \ker \phi$ and $\psi(g)>0$, or
			\item $g\in \ker \psi$ and $g>1$.
		\end{itemize}
		It is easy to see then that $>_{\alpha}$ is an order of type $\alpha$.

	\end{proof}

\end{document}